\documentclass{amsart}
\usepackage[utf8]{inputenc}
\usepackage{graphicx}              % to include figures
\usepackage{amsmath}               % great math stuff
\usepackage{amsfonts}              % for blackboard bold, etc
\usepackage{amsthm}                % better theorem environments
\usepackage{amssymb}
\usepackage{shuffle}                 % for the symbol shuffle
\usepackage{tikz}  % to draw Young tableaux
\usetikzlibrary{shapes}
\usetikzlibrary{positioning}
\usepackage{xcolor}
\usepackage{float}
\usepackage{pifont}
\usepackage[margin=1in]{geometry}
\newtheorem{thm}{Theorem}

\newtheorem{prop}[thm]{Proposition}
\newtheorem{cor}[thm]{Corollary}

\newtheorem{defn}[thm]{Definition}

\numberwithin{thm}{section}

\newcommand{\ZZ}{\mathbb{Z}}      % for Integers

\newcommand{\Des}{\mathrm{Des}}
\newcommand{\Inv}{\mathrm{Inv}}
\newcommand{\Adj}{\mathrm{Adj}}
\newcommand{\inv}{\mathrm{\Lambda}}
\newcommand{\Id}{\mathfrak{1}}

\title{A Counter Example to the Shuffle Compatiblity Conjecture}
\author[Kantarci Oguz]{Ezgi Kantarci O\u{g}uz}
\address{Department of Mathematics, University of Southern California, 3620 South Vermont Avenue, Los Angeles, CA 90089-2532, U.S.A.}
\email{kantarci@usc.edu}
\date{June 2018}

\begin{document}
\maketitle 
\begin{abstract} The shuffle product has a connection with several useful permutation statistics such as descent and peak, and corresponds to the multiplication operation in the corresponding descent and peak algebras. In their recent work, Gessel and Zhuang formalized the notion of shuffle-compatibility and studied various permutation statistics from this viewpoint. They further conjectured that any shuffle compatible permutation statistic is a descent statistic. In this note we construct a counter-example to this conjecture.
\end{abstract}

\section{Introduction}

A permutation $\sigma=\sigma_1 \sigma_2 \cdots \sigma_n$ of size $|\sigma|=n$ is a sequence of $n$ distinct integers. We denote the permutations of size $n$ with $P_n$, and set $P=\bigcup_n P_n$. Two permutations $\sigma$ and $\phi$ of the same size are said to be \emph{equivalent} if they have the same relative order, denoted $\sigma \sim \phi$. For example $1342\sim2894$, but they are not equivalent to $2891$, as the order of the first and last integers is different. Note that every permutation of size $n$ is equivalent to exactly one permutation of the numbers $1,2,\ldots,n$.

Two permutations said to be \emph{disjoint} if they do not share a number.  Given two disjoint permutations $\sigma$ and $\phi$, a \emph{shuffle} of $\sigma$ and $\phi$ is a permutation of size $|\sigma|+|\phi|$ that contains both $\sigma$ and $\phi$ as subsequences. We denote the set of shuffles of $\sigma$ and $\phi$ with $\sigma \shuffle \phi$.

A function $st$ on permutations is said to be a \emph{permutation statistic} if $st(\sigma)=st(\phi)$ whenever $\sigma \sim \phi$. Some examples of permutation statistics are defined below:

\begin{eqnarray*}
\Id(\sigma)&=&1,\\
\Des(\sigma)&=&\{i \mid \sigma_i>\sigma_{i+1}\},\\
\Inv(\sigma)&=&\{(i,j) \mid i<j \text{ and } \sigma_i>\sigma_j\}.\\
\end{eqnarray*}

\begin{defn}[\cite{MR3810249}] A permutation statistic $st$ is said to be shuffle compatible if for all disjoint permutations $\sigma$ and $\phi$, the multiset $\{\{st(\gamma)|\gamma \in \sigma \shuffle \phi\}\}$ depends only on $st(\sigma), st(\phi), |\sigma|$ and $|\phi|$.
\end{defn}
 From the examples above, $\Id$ and $\Des$ are shuffle compatible, whereas $\Inv$ is not. A recent paper by Gessel and Zhuang \cite{MR3810249} provides an in-depth exploration of shuffle compatible permutation statistics. They conjecture that any shuffle compatible permutation statistic $st$ is a \emph{descent statistic}, meaning if $\sigma,\phi \in P_n$ satisfy $\Des(\sigma)=\Des(\phi)$, then $st(\sigma)=st(\phi)$. In this note we will construct a permutation statistic that is shuffle compatible, but not a descent statistic. 
 
 \begin{prop} Let $st$ be a shuffle compatible statistic. For $|\sigma|=|\phi| < 4$, $\Des(\sigma)=\Des(\phi)$ implies $st(\sigma)=st(\phi)$.
\end{prop}
\begin{proof}
For sizes $1$ and $2$, any two permutations with the same descent set are equivalent, so there is nothing to show. Let us focus on size $3$. As these are permutation statistics, we can limit our attention to permutations of $1,2,3$. There are two pairs of non-equivalent permutations with the same descent set: $213-312$ and $231-132$. The calculations below show that $st(213)=st(312)$ and $st(231)=st(132)$.
\begin{eqnarray*}
\{\{st(\sigma)\mid \sigma \in 12 \shuffle 3\}\}&=&\{\{st(\sigma)\mid|\sigma \in 13 \shuffle 2\}\}, \\
\Rightarrow \{\{st(123),st(132),st(312)\}\}&=&\{\{st(132),st(123),st(213)\}\},\\
\{\{st(\sigma)\mid \sigma \in 23 \shuffle 1\}\}&=&\{\{st(\sigma)\mid|\sigma \in 13 \shuffle 2\}\}, \\
\Rightarrow \{\{st(231),st(213),st(123)\}\}&=&\{\{st(132),st(123),st(213)\}\}.
\end{eqnarray*} \end{proof}

This proposition shows that the minimum size we can have permutations that have the same descent set, but different values for some shuffle compatible statistic is $4$. 

Let $\sigma$ be a permutation of $a_1<a_2<a_3<a_4.$ Set
\begin{eqnarray*}
\Inv_{12}(\sigma)&=&\begin{cases}
1 & \text{ if } a_1 $ is to the left of $ a_2,\\
-1 & \text{ otherwise},
\end{cases}\\
\Adj_{34}(\sigma)&=&\begin{cases}
1 & \text{ if } a_3 $ is adjacent to $ a_4,\\
-1 & \text{ otherwise,}
\end{cases}\\
\inv(\sigma)&=& \Inv_{21}(\sigma) \cdot \Adj_{34}(\sigma).
\end{eqnarray*}

For example, $\sigma=2413$ has $\Inv_{12}(\sigma)=-1$ and $\Adj_{34}(\sigma)=-1$, so $\inv(2413)=-1 \cdot -1=1$. For $\phi=1423$,  $\Inv_{12}(\phi)=1$ and $\Adj_{34}(\phi)=-1$, so $\inv(1423)=-1 \cdot -1=1$.
%\begin{eqnarray*}
%S_4^-&=&\{ 1324,1423,2134,2143,2341,2431,3124,3142,3421,4123,4132, 4321\},\\
%S_4^+&=&\{ 1234,1243,1342,1432,2314,2413,3214,3241,3412,4213,4231,4312\}.
%\end{eqnarray*}

\begin{defn} We define a permutation statistic $\Psi:P\longrightarrow\ZZ$ as follows:
\begin{equation*}
    \Psi(\sigma)=\begin{cases}
    \inv(\sigma) & |\sigma|=4  ,\\
    1 & \text{otherwise}.
    \end{cases}
\end{equation*}
\end{defn}

\begin{prop} The function $\Psi$ is not a descent statistic.
\end{prop}
\begin{proof} For $\sigma=2413$ and $\phi=1423$, $\Des(\sigma)=\Des(\phi)=\{2\}$, but $\Psi(\sigma)=1 \neq \Psi(\phi)=-1$.
\end{proof}

\begin{thm} The function $\Psi$ is shuffle compatible.
\end{thm}
\begin{proof} Let $\sigma$ and $\phi$ be two permutations. Note that if $|\sigma|+|\phi|\neq 4$, the multiset $\{\{\Psi(\gamma)|\gamma \in \sigma \shuffle \phi\}\}$ contains only $1$s, and the number of those depends only on $|\sigma|$ and $|\phi|$. So we can focus on when $|\sigma|+|\phi|=4$. As we are working with a permutation statistic, it is enough to verify the result when the positive integers used are $1,2,3$ and $4$. Further note that the shuffle operation is symmetric, and $\Psi$ is symmetric under exchanging $3$ and $4$. 

\emph{Case 1: $|\sigma|=3$, $|\phi|=1$.} We claim that in this case $\{\{\Psi(\gamma)|\gamma \in \sigma \shuffle \phi\}\}=\{\{-1,-1,1,1\}\}$. If $\phi=3$, then independent of the placement of $4$, of the four elements of $\sigma \shuffle \phi$, exactly two have $3$ and $4$ adjacent, and the order of $1$ and $2$ is the same for all of them, so the claim holds. Same argument applies for the case $\phi=4$ by symmetry. The 6 other possibilities are illustrated at Table \ref{tab:1}, left.

\emph{Case 2: $|\sigma|=|\phi|=2$.} We claim that  $\{\{\Psi(\gamma)|\gamma \in \sigma \shuffle \phi\}\}=\{\{-1,-1,-1,1,1,1\}\}$. As exchanging $3$ and $4$ does not alter the $\Psi$ value, there are 6 possible pairings we need to consider, all illustrated in Table \ref{tab:1}, right. \end{proof}

\begin{table}[ht]
    \centering
    \begin{tabular}{c c}
    \begin{tabular}{|c|c|c|}
\hline
&$\Psi=+1$&$\Psi=-1$\\
\hline
$134\shuffle 2$ & $1234,1342$ & $1324, 2134$  \\
$314\shuffle 2$ & $2314,3214$ & $3124, 3142$ \\
$341\shuffle 2$ & $3412,3241$ & $3421, 2341$ \\
$234\shuffle 1$ & $1234,2314$ & $2341, 2134$  \\
$324\shuffle 1$ & $3241,3214$ & $3124, 1324$ \\
$342\shuffle 1$ & $3412,1342$ & $3421, 3142$ \\
\hline
\end{tabular} &    \begin{tabular}{|c|c|c|}
\hline
&$\Psi=+1$&$\Psi=-1$\\
\hline
$12\shuffle 34$ & $1234,1342,3412$ & $1324,3124,3142$ \\
$13\shuffle 24$ & $1234,1243,2413$ & $2134,2143,1324 $ \\
$13\shuffle 42$ & $4213, 1432,1342 $ & $4132,4123,1423 $ \\
$21\shuffle 34$ & $2314,3241,3214 $ & $2134,2341,3421 $  \\
$23\shuffle 14$ & $2314,1243,1234 $ & $1423,2134,2143 $ \\
$23\shuffle 41$ & $4213,2413,4231 $ & $4123,2341,2431 $ \\
\hline
\end{tabular}\\
    \end{tabular}
\vspace{3mm}
\caption{$\Psi$ values of shuffles of pairs $\gamma$ and $\phi$.}
    \label{tab:1}
\end{table}
\begin{cor} Conjecture 6.7 from \cite{MR3810249} is incorrect.
\end{cor}

Note that this counter example mainly depends on how $\Psi$ acts on permutations of size $4$. Even though taking only two equivalence classes at size $4$ limits our options for larger sizes, it does not force the existence of only one equivalence class at each level, that is just selected for simplicity.

A question that arises from this counter example is whether it is possible to refine the conjecture by adding extra conditions to ensure the resulting statistics only depend on descent. One such result was recently proved by Grinberg in \cite{Grinbe}: 

\begin{defn}[\cite{Grinbe}] A permutation statistic $st$ is \emph{left shuffle compatible} if the multiset $\{\{st(\gamma)|\gamma \in \sigma \shuffle \phi \text{ and } \gamma_1=\sigma_1\}\}$ depends only on $st(\sigma),st(\phi),|\sigma|$ and $|\phi|$.
\end{defn}

\begin{prop}[\cite{Grinbe}] Any shuffle compatible and left shuffle compatible statistic is a descent statistic.
\end{prop}

Note that our counter-example does not violate this result, as it is not left shuffle compatible:
\begin{eqnarray*}
\{\{\Psi(\gamma)\mid\gamma \in 12 \shuffle 34 \text{ and } \gamma_1=1\}\}&=&\{\{\Psi(1234),\Psi(1342),\Psi(1324)\}\}=\{\{1,1,-1\}\},\\
\{\{\Psi(\gamma)\mid\gamma \in 34 \shuffle 12 \text{ and } \gamma_1=3\}\}&=&\{\{\Psi(3124),\Psi(3142),\Psi(3412)\}\}=\{\{1,-1,-1\}\}.
\end{eqnarray*}

There are examples of shuffle compatible descent statistics that are not left shuffle compatible, so the above proposition does not offer a complete characterization. Nevertheless, it might be the best result to be obtained on the subject. 

\bibliographystyle{plain}
\bibliography{main}
\end{document}